\documentclass[a4paper,11pt]{amsart}
\usepackage{a4wide}
\usepackage{enumerate}
\usepackage[utf8]{inputenc}
\usepackage{url}
\usepackage{pst-node}
\usepackage{booktabs}          
\usepackage{float}             
\usepackage{hyperref}
\IfFileExists{\jobname.TEX}{}{\usepackage{srcltx}}
\usepackage{array}
\usepackage{comment}

\numberwithin{equation}{section}

\newcommand{\field}[1]{\mathbb{{#1}}}

\newcommand{\Q}{\field{{Q}}}
\newcommand{\K}{\field{{K}}}

\newcommand{\disc}{\Delta_\K}
\newcommand{\rdisc}{\delta_\K}
\newcommand{\nK}{n_\K}

\def\lamb{w}

\newtheorem{theorem}{Theorem}[section]
\newtheorem{lemma}[theorem]{Lemma}
\newtheorem*{lemma*}{Lemma}

\theoremstyle{remark}
\newtheorem{remark}{Remark}
\newtheorem*{remark*}{Remark}
\newtheorem*{acknowledgements}{Acknowledgements}

\allowdisplaybreaks[4]

\begin{document}
\title[Explicit prime ideal theorem, II]{Explicit versions of the prime ideal theorem for Dedekind zeta functions under GRH, II}

\author[L.~Grenié]{Loïc Grenié}
\address[L.~Grenié]{Dipartimento di Ingegneria Gestionale, dell'Informazione
         e della Produzione\\
         Università degli Studi di Bergamo\\
         viale Marconi 5\\
         24044 Dalmine\\
         Italy}
\email{loic.grenie@gmail.com}

\author[G.~Molteni]{Giuseppe Molteni}
\address[G.~Molteni]{Dipartimento di Matematica\\
         Università di Milano\\
         via Saldini 50\\
         20133 Milano\\
         Italy}
\email{giuseppe.molteni1@unimi.it}

\keywords{} \subjclass[2010]{Primary 11R42, Secondary 11Y70}

\date{\today. File name: {\tt \jobname.tex}}

\begin{abstract}
We have recently proved several explicit versions of the prime ideal
theorem under GRH. Here we further explore the method, in order to deduce its
strongest consequence for the case where $x$ diverges.
\end{abstract}

\maketitle

\section{Introduction}\label{sec:A1}
For a number field $\K$ we denote by
\begin{itemize}
\item[] $\nK$ its dimension,
\item[] $\disc$ the absolute value of its discriminant,
\item[] $\rdisc:=(\disc)^{(1/\nK)}$ its root discriminant,
\item[] $r_1$ the number of its real places,
\item[] $r_2$ the number of its imaginary places,
\item[] $d_\K := r_1+r_2-1$.
\end{itemize}
In \cite{GrenieMolteni2} we use a two step process to prove explicit versions
of the prime ideal theorem under GRH: first we prove a bound for
$|\psi_\K(x)-x|$ depending on a parameter $T$ to be fixed later
(\cite[Theorem~1.1]{GrenieMolteni2}), then we prove several formulas based on
some choices for $T$ (\cite[Corollaries~1.2 and 1.3]{GrenieMolteni2}). A scheme
to produce explicit versions of the prime ideal theorem for number fields has
been proved by Lagarias and Odlyzko in~\cite{LagariasOdlyzko} and recently
Winckler computed the effective constants in~\cite{Winckler}. In this paper, we
reuse Theorem~1.1 of \cite{GrenieMolteni2} with an additional parameter called
$\kappa$ producing the general result in Theorem~\ref{th:B2.5}, we then choose
$\kappa$ and $T$ in such a way as to obtain the best possible asymptotic
expansion for $x\to{+\infty}$. By doing so we obtain a formula that is not too
far from the best possible bound for $|\psi_\K(x)-x|$ which can be proved by
using this method.

We recall that the Lambert-$W$ function is the function such that
$\forall x\geq 0$, $W(x)e^{W(x)}=x$.
\begin{theorem}\label{th:B1.1}
Assume GRH. Let $x\geq 3$,
\begin{align*}
\lamb
 &:=
 W\Big(\frac{e^{\sqrt{5}}}{2\pi}\rdisc\Big[\frac{(\sqrt{5}-1)\pi\sqrt{x}}{2\nK}
          + 21.3270
          + \frac{33.3542}{\nK}
     \Big]\Big),                                                              \\
T
 &:=8.2822+\frac{1}{\lamb}\Big[\frac{(\sqrt{5}-1)\pi\sqrt{x}}{2\nK}
          + 21.3270
          + \frac{33.3542}{\nK}\Big],                                         \\
%
%
\epsilon_\K(x,T)
 &:=\max\Big(0, d_\K\log x - 3.6133\nK\frac{\sqrt{x}}{T}\Big).
\end{align*}
Then
\begin{multline}\label{eq:A1.1}
|\psi_\K(x)-x| \leq                                                           \\
\frac{\sqrt{x}}{\pi}\Big[\Big(\frac{1}{2}\log^2\big(e^{\lamb+1}+33.5251\rdisc\big)
                             -\frac{1}{2}\log^2\rdisc
                             +3.9792\log\rdisc-3.4969\Big)\nK+25.5362\Big]    \\
 +1.0155\log\disc
 -2.1042\nK
 +8.8590
 +\epsilon_\K(x,T).
\end{multline}
Moreover we also have
\begin{multline}\label{eq:A1.2}
|\psi_\K(x)-x|
\leq (2.2543\sqrt{x}+1.0155)\log\disc + (0.9722\sqrt{x}-2.1042)\nK \\
    + \frac{x}{10}
    + 9.0458\sqrt{x}
    + 7.0320
    + \epsilon_\K(x,10).
\end{multline}
\end{theorem}
\noindent
The choice of $T$ we have made to deduce Theorem~\ref{th:B1.1} from
Theorem~\ref{th:B2.5} gives the best coefficients for all terms in the
asymptotic expansion, down to the term of order $\sqrt{x}$.
This choice is not too far from the best our method can achieve, even for
finite $x$; in other words, the $T$ we choose in Theorem~\ref{th:B1.1}
is not too far from the optimal $T$ for Theorem~\ref{th:B2.5}.
Note that the values of the other parameters we fix in the proofs of
Theorem~\ref{th:B1.1} and Theorem~\ref{th:B2.5} affect the term of order
$\sqrt{x}$ of the asymptotic expansion.

Inequality~\eqref{eq:A1.2} is a kind of Chebyshev bound, which has the
interesting property that the linear term has a coefficient independent of the
field. It is better than~\eqref{eq:A1.1} when $x$ is very small with respect to
$\rdisc$.

\subsection*{Asymptotic expansions}
We discuss the asymptotic expansions of~\eqref{eq:A1.1} when $x$ diverges and
$\K$ is fixed. We have, as $t\to{+\infty}$
\[
W(t) = \log t-\log\log t + \frac{\log\log t}{\log t}
+ O\Big(\Big(\frac{\log\log t}{\log t}\Big)^2\Big)
\]
and, even though we will not use it,
\[
\forall t\geq e,\quad
W(t) \leq \log t-\log\log t + 1.024\frac{\log\log t}{\log t}
\]
so that the asymptotic expansion we are computing is not too far from an upper
bound. From the first expansion, we deduce
\[
\frac{1}{2}W(t)^2+W(t)
=
\frac{1}{2}\log^2 t - \log t\log\log t + \log t + \frac{1}{2}(\log\log t)^2
+ o(1).
\]
We have $\lamb\sim \frac{1}{2}\log x$ when $x$ diverges, thus
\[
\log^2\big(e^{\lamb+1}+33.5251\rdisc\big)
= \Big(\lamb+1+\frac{33.5251\rdisc}{e^{\lamb+1}}+O\Big(\frac{1}{e^{2w}}\Big)\Big)^2
=  \lamb^2+2w+1+O\Big(\frac{\log x}{\sqrt{x}}\Big).
\]
We thus have, taking $\nu:=\frac{\sqrt{5}-1}{2}e^{\sqrt{5}}$
\begin{align*}
\frac{1}{2}\log^2&\big(e^{\lamb+1}+33.5251\rdisc\big)-\frac{1}{2}\log^2\rdisc \\
=&
 \frac{1}{2}\lamb^2+\lamb+\frac{1}{2}-\frac{1}{2}\log^2\rdisc+o(1)            \\
=&
 \frac{1}{2}\Big[\log\rdisc
                + \frac{1}{2}\log x
                + \log\Big(\frac{\nu}{2\nK}\Big)
             \Big]^2                                                          \\
&-      \Big[\log\rdisc
                + \frac{1}{2}\log x
                + \log\Big(\frac{\nu}{2\nK}\Big)
             \Big]
             \Big[\log\log x-\log 2+\frac{2\log\rdisc}{\log x}
                + \frac{2}{\log x}
                  \log\Big(\frac{\nu}{2\nK}\Big)
             \Big]                                                            \\
&+           \log\rdisc
                + \frac{1}{2}\log x
                + \log\Big(\frac{\nu}{2\nK}\Big)
 +           \frac{1}{2}\Big[\log\log x-\log 2\Big]^2
 +           \frac{1}{2}
 -           \frac{1}{2}\log^2\rdisc
 + o(1)                                                                       \\
=&
 \frac{1}{8}\log^2x - \frac{1}{2}\log x\log\log x
 +\frac{1}{2}\Big[\log\rdisc
                + \log\Big(\frac{e\nu}{\nK}\Big)
             \Big]\log x
 +\frac{1}{2}(\log\log x)^2                                                   \\
&-           \Big[\log\rdisc
                + \log\Big(\frac{\nu}{\nK}\Big)
             \Big]\log\log x
 + \frac{1}{2}\log^2\Big(\frac{\nu}{\nK}\Big)
 + \log\Big(\frac{\nu}{\nK}\Big)\log\rdisc
 + \frac{1}{2}
 + o(1).
\end{align*}
Thus the right hand side of~\eqref{eq:A1.1} is
\begin{align*}
&\nK\frac{\sqrt{x}}{2\pi}
\Big[
      \frac{1}{4}\log^2x - \log x\log\log x
      + \Big[\log\rdisc
             + \log\Big(\frac{e\nu}{\nK}\Big)
        \Big]\log x
      + (\log\log x)^2                      \\
&\qquad\quad
      -2\Big[\log\rdisc
             + \log\Big(\frac{\nu}{\nK}\Big)
        \Big]\log\log x
      + \Big(2\log\Big(\frac{\nu}{\nK}\Big)
             + 7.9584
        \Big)\log\rdisc
      + \log^2\Big(\frac{\nu}{\nK}\Big)
      - 5.9938
  \Big]                                     \\
&+ 25.5362\frac{\sqrt{x}}{\pi}
   + o(\sqrt{x}).
\end{align*}
Since $e\nu\simeq 15.7187\dots$ the coefficient of
$\nK\frac{\sqrt{x}}{2\pi}\log x$ is lower than $\log\rdisc$ if $\nK\geq 16$.

\noindent{}We have verified that the first five terms in the asymptotic
expansion cannot be improved by any choice of the parameters. On the other
hand, the sixth term contains the constants $7.9584$, $-5.9938$ and $25.5362$
which can be changed acting on the parameters.

\noindent{}The constants hidden in the $o(.)$ terms are unfortunately not
uniform in $\K$ and are not even controlled by a linear bound in $\nK$ and
$\log\disc$: for instance, the rather innocent looking $\frac{\nK\sqrt{x}}{T}$
is asymptotic to $\frac{2}{(\sqrt{5}-1)\pi}\nK^2\log x$.

\noindent{}To facilitate the comparison with earlier results, we reorganize
this asymptotic expansion in a form similar to Lagarias and Odlyzko's or
Oesterlé's results. In this form, the right hand side of~\eqref{eq:A1.1} is
\begin{align*}
&\!\!\!\!\frac{\sqrt{x}}{2\pi}
   \Big[\log x
       - 2\log\log x
       + 2\log\Big(\frac{\nu}{\nK}\Big)
       + 7.9584
    \Big]\log\disc                                                         \\
&+\frac{\sqrt{x}}{8\pi}
    \Big[\log^2 x
       - 4\log x\log\log x
       + 4\log\Big(\frac{e\nu}{\nK}\Big)\log x
       + 4(\log\log x)^2
       - 8\log\Big(\frac{\nu}{\nK}\Big)\log\log x                          \\
&\qquad
       + 4 \log^2\Big(\frac{\nu}{\nK}\Big)
       + 4
       - 23.9752
    \Big]\nK
+ 25.5362\frac{\sqrt{x}}{\pi}
+ o(\sqrt{x})                                                              \\
=&\frac{\sqrt{x}}{2\pi}
    \Big[\log  \Big(\frac{e^2\nu^2}{\nK^2}\,\frac{x}{\log^2 x}\Big)
       + 5.9584
    \Big]\log\disc                                                         \\
 &+\frac{\sqrt{x}}{8\pi}
    \Big[\log^2\!\Big(\frac{x}{\log^2 x}\Big)
       + 4\log\!\Big(\frac{e\nu}{\nK}\Big)
          \log\!\Big(\frac{x}{\log^2 x}\Big)
       + 4\log^2\!\Big(\frac{e\nu}{\nK}\Big)
       - 8\log\!\Big(\frac{\nu}{\nK\log x}\Big)
       - 27.9752
    \Big]\nK                                                               \\
 &+25.5362\frac{\sqrt{x}}{\pi}
+ o(\sqrt{x})                                                              \\
=&\frac{\sqrt{x}}{2\pi}
    \Big[\log  \Big(\frac{e^2\nu^2}{\nK^2}\,\frac{x}{\log^2 x}\Big)
       + 5.9584
    \Big]\log\disc                                                         \\
 &+\frac{\sqrt{x}}{8\pi}
    \Big[\log^2\Big(\frac{e^2\nu^2}{\nK^2}\,\frac{x}{\log^2 x}\Big)
       - 4\log\Big(\frac{e^2\nu^2}{\nK^2}\,\frac{1}{\log^2 x}\Big)
       - 19.9752
    \Big]\nK
+ 25.5362\frac{\sqrt{x}}{\pi}
+ o(\sqrt{x}).
\end{align*}

\smallskip
As $\rdisc$ diverges~\eqref{eq:A1.1} is not very efficient, but still gives
something similar to~\eqref{eq:A1.2}.

\subsection*{Numerical experiments}
In \cite{GrenieMolteni2} we prove the following results. First in
Corollary~1.2:
\begin{equation}\label{eq:A1.3}
\forall x\geq 100,\quad|\psi_\K(x)-x| \leq
\sqrt{x}\Big[\Big(\frac{\log   x}{2\pi} + 2\Big)\log\disc
           + \Big(\frac{\log^2 x}{8\pi} + 2\Big)\nK
        \Big].
\end{equation}
Then in Corollary~1.3:
\begin{multline}\label{eq:A1.4}
\forall x\geq 3,\phantom{000}\quad
|\psi_\K(x)-x|
\leq  \sqrt{x}\Big[\Big(\frac{1}{2\pi}\log\Big(\frac{18.8\,x}{\log^2 x}\Big)
                        + 2.3
                   \Big)\log\disc                                         \\
                  + \Big(\frac{1}{8\pi}\log^2\Big(\frac{18.8\,x}{\log^2 x}\Big)
                          + 1.3
                     \Big)\nK
                  + 0.3\log x
                  + 14.6
              \Big]
\end{multline}
and
\begin{multline}\label{eq:A1.5}
\forall x\geq 2000,\quad
|\psi_\K(x)-x|
\leq  \sqrt{x}\Big[\Big(\frac{1}{2\pi}\log\Big(\frac{x}{\log^2 x}\Big)
                        + 1.8
                   \Big)\log\disc                                        \\
                  +\Big(\frac{1}{8\pi}\log^2\Big(\frac{x}{\log^2 x}\Big)
                        + 1.1
                   \Big)\nK
                  + 1.2\log x
                  + 10.2
              \Big].
\end{multline}
We compare the upper bound~\eqref{eq:A1.1} to these three formulas for
several values of $\nK$ and four discriminants for each $\nK$. We test
totally real and totally imaginary fields for the minimal discriminants
allowed by Odlyzko's Table 3 in~\cite{OdlyzkoTables} and for their squares.
In each table we indicate the minimal $x$
after which Formula~\eqref{eq:A1.1} is better than the corresponding formula.
One observes that \eqref{eq:A1.1} is always better than \eqref{eq:A1.4},
nearly always better than \eqref{eq:A1.3} (except for quadratic fields) and
most of the times better than \eqref{eq:A1.5}. The best between
\eqref{eq:A1.1} and \eqref{eq:A1.2} is always better than
(\ref{eq:A1.3}--\ref{eq:A1.5}) except for the case of quadratic fields
in
Formula~\eqref{eq:A1.3}.\\
\bigskip\\
\noindent%
{%
\def\header{$\nK$ & \multicolumn{1}{c}{$\disc$} & \eqref{eq:A1.3} & \eqref{eq:A1.4} & \eqref{eq:A1.5}\\\hline}%
\def\line[#1, #2, #3, #4, #5]{#1 & #2 & #3 & #4 & #5\\}%
\def\sheader{\multicolumn{1}{c}{$\disc$} & \eqref{eq:A1.3} & \eqref{eq:A1.4} & \eqref{eq:A1.5}\\\hline}%
\def\sline[#1, #2, #3, #4, #5]{#2 & #3 & #4 & #5\\}%
\centerline{From when does \eqref{eq:A1.1} get better than (\ref{eq:A1.3}--\ref{eq:A1.5})}
\begin{tabular}{@{}c@{}c@{}c@{}}
& real & imaginary
\\
minimal
&
\begin{tabular}{r>{\(}l<{\)}rrr}
\header
\line[  2, 4.9535              , 187929, 3, 2000]
\line[  6, 2.9169\cdot 10^{5}  ,    107, 3, 2000]
\line[ 10, 2.3927\cdot 10^{11} ,    100, 3, 2000]
\line[ 20, 6.5601\cdot 10^{27} ,    100, 3, 2000]
\line[ 50, 7.1245\cdot 10^{81} ,    100, 3, 2425]
\line[100, 1.5472\cdot 10^{177},    100, 3, 2713]
\line[200, 8.0911\cdot 10^{374},    100, 3, 2851]
\end{tabular}
\quad&\quad
\begin{tabular}{>{\(}l<{\)}rrr}
\sheader
\sline[  2, 2.9633              , 445897, 3, 2000]
\sline[  6, 9.3896\cdot 10^{3}  ,    106, 3, 2000]
\sline[ 10, 1.8967\cdot 10^{8}  ,    100, 3, 2000]
\sline[ 20, 1.7076\cdot 10^{20} ,    100, 3, 2000]
\sline[ 50, 2.8528\cdot 10^{59} ,    100, 3, 2306]
\sline[100, 3.0629\cdot 10^{128},    100, 3, 2663]
\sline[200, 2.1888\cdot 10^{271},    100, 3, 2843]
\end{tabular}
\\
\\
square
&
\begin{tabular}{r>{\(}l<{\)}rrr}
\header
\line[  2, 2.4538\cdot 10^{1}  , 25000, 3, 2000]
\line[  6, 8.5086\cdot 10^{10} ,   100, 3, 2000]
\line[ 10, 5.7250\cdot 10^{22} ,   100, 3, 2000]
\line[ 20, 4.3035\cdot 10^{55} ,   100, 3, 2074]
\line[ 50, 5.0759\cdot 10^{163},   100, 3, 2597]
\line[100, 2.3937\cdot 10^{354},   100, 3, 2757]
\line[200, 6.5467\cdot 10^{749},   100, 3, 2830]
\end{tabular}
&
\begin{tabular}{>{\(}l<{\)}rrr}
\sheader
\sline[  2, 8.7813              , 81922, 3, 2000]
\sline[  6, 8.8164\cdot 10^{7}  ,   100, 3, 2000]
\sline[ 10, 3.5975\cdot 10^{16} ,   100, 3, 2000]
\sline[ 20, 2.9158\cdot 10^{40} ,   100, 3, 2000]
\sline[ 50, 8.1386\cdot 10^{118},   100, 3, 2532]
\sline[100, 9.3814\cdot 10^{256},   100, 3, 2745]
\sline[200, 4.7910\cdot 10^{542},   100, 3, 2844]
\end{tabular}
\end{tabular}
\vfill\pagebreak
\centerline{From when does the best of \eqref{eq:A1.1} and
\eqref{eq:A1.2} get better than (\ref{eq:A1.3}--\ref{eq:A1.5})}
\begin{tabular}{@{}c@{}c@{}c@{}}
& real & imaginary
\\
minimal
&
\begin{tabular}{r>{\(}l<{\)}rrr}
\header
\line[  2, 4.9535              , 187929, 3, 2000]
\line[  6, 2.9169\cdot 10^{5}  ,    100, 3, 2000]
\line[ 10, 2.3927\cdot 10^{11} ,    100, 3, 2000]
\line[ 20, 6.5601\cdot 10^{27} ,    100, 3, 2000]
\line[ 50, 7.1245\cdot 10^{81} ,    100, 3, 2000]
\line[100, 1.5472\cdot 10^{177},    100, 3, 2000]
\line[200, 8.0911\cdot 10^{374},    100, 3, 2000]
\end{tabular}
\quad&\quad
\begin{tabular}{>{\(}l<{\)}rrr}
\sheader
\sline[  2, 2.9633              , 445897, 3, 2000]
\sline[  6, 9.3896\cdot 10^{3}  ,    100, 3, 2000]
\sline[ 10, 1.8967\cdot 10^{8}  ,    100, 3, 2000]
\sline[ 20, 1.7076\cdot 10^{20} ,    100, 3, 2000]
\sline[ 50, 2.8528\cdot 10^{59} ,    100, 3, 2000]
\sline[100, 3.0629\cdot 10^{128},    100, 3, 2000]
\sline[200, 2.1888\cdot 10^{271},    100, 3, 2000]
\end{tabular}
\\
\\[-1.7ex]
square
&
\begin{tabular}{r>{\(}l<{\)}rrr}
\header
\line[  2, 2.4538\cdot 10^{1}  , 25000, 3, 2000]
\line[  6, 8.5086\cdot 10^{10} ,   100, 3, 2000]
\line[ 10, 5.7250\cdot 10^{22} ,   100, 3, 2000]
\line[ 20, 4.3035\cdot 10^{55} ,   100, 3, 2000]
\line[ 50, 5.0759\cdot 10^{163},   100, 3, 2000]
\line[100, 2.3937\cdot 10^{354},   100, 3, 2000]
\line[200, 6.5467\cdot 10^{749},   100, 3, 2000]
\end{tabular}
&
\begin{tabular}{>{\(}l<{\)}rrr}
\sheader
\sline[  2, 8.7813              , 81922, 3, 2000]
\sline[  6, 8.8164\cdot 10^{7}  ,   100, 3, 2000]
\sline[ 10, 3.5975\cdot 10^{16} ,   100, 3, 2000]
\sline[ 20, 2.9158\cdot 10^{40} ,   100, 3, 2000]
\sline[ 50, 8.1386\cdot 10^{118},   100, 3, 2000]
\sline[100, 9.3814\cdot 10^{256},   100, 3, 2000]
\sline[200, 4.7910\cdot 10^{542},   100, 3, 2000]
\end{tabular}
\end{tabular}
}
\begin{acknowledgements}
We wish to thank Alberto Perelli, for his valuable remarks and comments,
and the referee for her/his careful reading.
All computations in this paper has been made using PARI/GP~\cite{PARI2}.
The authors are members of the GNSAGA INdAM group.
\end{acknowledgements}

\section{Reproofs}\label{sec:A2}
We reprove the results of~\cite{GrenieMolteni1} and \cite{GrenieMolteni2}
adding an additional parameter to the main result of~\cite{GrenieMolteni2} and
with a couple more digits. The methods of proof are the same and will thus not
be repeated.

We recall that $r_\K$ is defined as the constant such that
\[
\frac{\zeta'_\K}{\zeta_\K}(s) = \frac{r_1+r_2-1}{s} + r_\K + O(s)\quad \text{as $s\to  0$}.
\]
\begin{lemma}\label{lem:A2.1}
Assume GRH. One has
\[
|r_\K|\leq  1.0155\log \disc - 2.1042 \nK + 8.3423 - e_\K
\]
where
\[
e_\K :=
\left\{
\begin{array}{ll}
4.4002 & \text{if }(r_1,r_2)=(1,0) \\
0.6931 & \text{if }(r_1,r_2)=(0,1) \\
0      & \text{otherwise}.
\end{array}
\right.
\]
\end{lemma}
\begin{proof}
The proof of Lemma~3.1 in~\cite{GrenieMolteni1} gives the general case.
Moreover, for $\Q$ we know that $r_\Q=\log 2\pi$. For imaginary quadratic
fields, an $r_2\log 2$ term can be restored in the proof of the aforementioned
lemma.
\end{proof}

\begin{lemma}\label{lem:A2.2}
Assume GRH. One has
\[
\sum_{|\gamma|\leq 5} \frac{1}{|\rho|} \leq 1.0111\log\disc - 1.6550 \nK + 7.0320.
\]
\end{lemma}
\begin{proof}
see Lemma~3.1 in~\cite{GrenieMolteni2}.
\end{proof}
As in \cite{GrenieMolteni2} we will denote
$W_\K(T):=\log\disc+\nK\log\big(\frac{T}{2\pi}\big)$; this is obviously not the
Lambert $W$ function, and we believe that there is no risk of confusion.
\begin{lemma}\label{lem:A2.3}
We have, for all $T\geq 5$,
\begin{subequations}
\begin{align}
\sum_{|\gamma|\leq T}1
&\leq \frac{T}{\pi}\Big(1+\frac{1.4427}{T}\Big)W_\K(T)
      - \frac{T}{\pi}\Big(1 - \frac{8.9250}{T}\Big)\nK
      + \frac{8.6542}{\pi},                                            \label{eq:A2.1a}\\
\sum_{|\gamma|\geq T}\frac{1}{|\rho|^2}
&\leq \Big(1 + \frac{2.8854}{T}\Big)\frac{W_\K(T)}{\pi T}
    +\Big(1 + \frac{18.6019}{T}\Big)\frac{\nK}{\pi T}
    +\frac{17.3084}{\pi T^2},                                          \label{eq:A2.1b}\\
\sum_{\substack{\rho\\|\gamma|\leq T}}\frac{\pi}{|\rho|}
&\leq \Big(\log\Big(\frac{T}{2\pi}\Big) + 3.9792\Big)\log \disc
     + \Big(\frac{1}{2}\log^2\Big(\frac{T}{2\pi}\Big) - 1.4969\Big)\nK
     + 25.5362.                                                        \label{eq:A2.1c}
\end{align}
\end{subequations}
\end{lemma}
%
\begin{proof}
These are just \textbf{First sum}, \textbf{Second sum} and \textbf{Third sum},
in \cite{GrenieMolteni2} with more digits. To prove \eqref{eq:A2.1c}, we use
Lemma~\ref{lem:A2.2} instead of Lemma~3.1 of \cite{GrenieMolteni2}.
\end{proof}
\begin{lemma}\label{lem:A2.4}
Let
\[
f_1(x) := \sum_{r=1}^{\infty}\frac{x^{1-2r}}{2r(2r-1)},
\qquad
f_2(x) := \sum_{r=2}^{\infty}\frac{x^{2-2r}}{(2r-1)(2r-2)},
\]
\[
R_{r_1,r_2}(x) := - (r_1+r_2-1)(x\log x-x)
                  + r_2(\log x +1)
                  - (r_1+r_2)f_1(x)
                  - r_2      f_2(x).
\]
Let $x\geq 3$, then
\[
-(r_1+r_2-1)\log x \leq R_{r_1,r_2}'(x) \leq -\delta_{(r_1,r_2),(1,0)}\log(1-x^{-2}) - \delta_{(r_1,r_2),(0,1)}\log(1-x^{-1})
\]
where $\delta_{(r_1,r_2),(a,b)}$ is $1$ if and only if both indices are equal
and $0$ otherwise.
\end{lemma}
\begin{proof}
See Lemma~2.2 in \cite{GrenieMolteni2}.
\end{proof}
We now restate Theorem~1.1 of \cite{GrenieMolteni2}.
\begin{theorem}\label{th:B2.5}
For every $x\geq 3$, $T\geq 5$ and $0<\kappa\leq 2$ we have:
\begin{align*}
\Big|\psi_\K(x)-x+\sum_{|\gamma|<T}\frac{x^{\rho}}{\rho}\Big|
\leq& \frac{\sqrt{x}}{\pi}
               \Big[\frac{2}{\kappa}+\frac{\kappa}{2}
                  + \frac{1.4427\kappa^2+3\kappa+11.5416}{2\kappa T}
                  + \frac{0.5915\kappa+4.3282}{T^2}
               \Big]W_\K(T)                                              \\
&+\frac{\sqrt{x}}{\pi}
               \Big[\frac{2}{\kappa} - \frac{\kappa}{2}
                  + \frac{8.9250\kappa^2 + 3\kappa + 74.4076}{2\kappa T}
                  + \frac{1.7702\kappa + 27.9029}{T^2}
               \Big]\nK                                                  \\
&+ \frac{\kappa x}{2T}
 +\frac{\sqrt{x}}{\pi}
               \Big[\frac{(1.3774\kappa^2 + 11.0190)\pi}{\kappa T}
                  + \frac{(0.4133\kappa + 8.2643)\pi}{T^2}
               \Big]                                                     \\
&+ |r_\K|
 + \tilde{\epsilon}_\K(x,T)
\end{align*}
where
\[
\tilde{\epsilon}_\K(x,T):=
\left\{
\begin{array}{ll}
-\log(1-x^{-2})                                           & \text{if }(r_1,r_2)=(1,0) \\
-\log(1-x^{-1})                                           & \text{if }(r_1,r_2)=(0,1) \\
\max\big(0, d_\K\log x - 3.6133\nK\frac{\sqrt{x}}{T}\big) & \text{otherwise}.
\end{array}
\right.
\]
\end{theorem}

\begin{proof}
The proof proceeds as for \cite[Theorem~1.1]{GrenieMolteni2}, but now we
choose $h=\pm\frac{\kappa x}{T}$ with $\kappa\in(0,2]$ instead of
$h=\pm\frac{2x}{T}$. We start from \cite[Inequality~(4.2)]{GrenieMolteni2}
which, given the small modification of Lemma~\ref{lem:A2.4} above, now reads:
\begin{multline*}
-(r_1+r_2-1)\log x\\
\leq
\frac{\psi^{(1)}_\K(x+h)-\psi^{(1)}_\K(x)}{h}
- \Big(
        x+\frac{h}{2}
        - \sum_{\rho}\frac{(x+h)^{\rho+1}-x^{\rho+1}}{h\rho(\rho+1)}
        - r_\K
  \Big)\\
\leq
-\delta_{(r_1,r_2),(1,0)}\log(1-x^{-2})-\delta_{(r_1,r_2),(0,1)}\log(1-x^{-1}).
\end{multline*}
As seen in \cite[Section 4]{GrenieMolteni2}, under GRH we have
\begin{align*}
\Big|\sum_{|\gamma|\geq T}\frac{(x+h)^{\rho+1}-x^{\rho+1}}{h\rho(\rho+1)}\Big|
\leq &
\sum_{|\gamma|\geq T}
 x^{\frac{3}{2}}\frac{\big(1+\frac{h}{x}\big)^{\frac{3}{2}}+1}{h|\rho(\rho+1)|}
\leq
A\frac{x^{\frac{3}{2}}}{h}\sum_{|\gamma|\geq T}\frac{1}{|\rho^2|},
\end{align*}
with $A:=1+\big(1+\frac{h}{x}\big)^{\frac{3}{2}}$ while
\[
\sum_{|\gamma|<T}\frac{(x+h)^{\rho+1}-x^{\rho+1}}{h\rho(\rho+1)}
=
\sum_{|\gamma|<T}\frac{x^\rho}{\rho}+
hx^{-1/2}\sum_{|\gamma|<T}w_\rho x^{i\gamma}
\]
with
\[
w_\rho := \frac{\big(1+\frac{h}{x}\big)^{\rho+1}-1-(\rho+1)\frac{h}{x}}
               {\rho(\rho+1)\big(\frac{h}{x}\big)^2}.
\]
For $h>0$, we take $h=\frac{\kappa x}{T}$ thus
\[
A  =  1+\Big(1+\frac{\kappa}{T}\Big)^{\frac{3}{2}}
 \leq 2+\frac{3\kappa}{2T}+\frac{3\kappa^2}{8T^2}.
\]
From Lemma~2.1 in~\cite{GrenieMolteni2} we know that $|w_\rho|\leq\frac{1}{2}$
hence by~\eqref{eq:A2.1b} and \eqref{eq:A2.1a}
\begin{align*}
\frac{\pi}{\sqrt{x}}
&
\Big|\sum_{\rho}\frac{(x+h)^{\rho+1}-x^{\rho+1}}{h\rho(\rho+1)}
     -\sum_{|\gamma|<T}\frac{x^\rho}{\rho}
\Big|\\
\leq&
  \Big[\frac{2}{\kappa}+\frac{3}{2T} +\frac{3\kappa}{8T^2}\Big]
                           \Big[
                            \Big(1 + \frac{2.8854}{T}\Big)\frac{W_\K(T)}{\pi T}
                          + \Big(1 + \frac{18.6019}{T}\Big)\frac{\nK}{\pi T}
                          + \frac{17.3084}{T}
                           \Big] \\
&  + \frac{\kappa}{2}      \Big[
                            \Big(1 + \frac{1.4427}{T}\Big)W_\K(T)
                          - \Big(1 - \frac{8.9250}{T}\Big)\nK
                          + \frac{8.6542}{T}
                           \Big]\\
=& \Big[\Big(\frac{2}{\kappa}+\frac{3}{2T} +\frac{3\kappa}{8T^2}\Big)\Big(1 + \frac{2.8854}{T}\Big) + \frac{\kappa}{2}\Big(1+\frac{1.4427}{T}\Big)
   \Big]W_\K(T)\\
 &+\Big[\Big(\frac{2}{\kappa}+\frac{3}{2T} +\frac{3\kappa}{8T^2}\Big)\Big(1 + \frac{18.6019}{T}\Big)
        - \frac{\kappa}{2}\Big(1 - \frac{8.9250}{T}\Big)
   \Big]\nK\\
 & + \Big(\frac{2}{\kappa}+\frac{3}{2T} +\frac{3\kappa}{8T^2}\Big)\frac{17.3084}{T}
   +4.3271\frac{\kappa}{T}.
\end{align*}
After some simplifications it becomes, for $T\geq 5$,
\begin{align}
\frac{\pi}{\sqrt{x}}
  \Big|\sum_{\rho}&\frac{(x+h)^{\rho+1}-x^{\rho+1}}{h\rho(\rho+1)}
       -\sum_{|\gamma|<T}\frac{x^\rho}{\rho}
  \Big|                                                        \label{eq:A2.2}\\
\leq&    \Big[\frac{2}{\kappa}+\frac{\kappa}{2}
        + \frac{1.4427\kappa^2+3\kappa+11.5416}{2\kappa T}
        + \frac{0.5915\kappa+4.3281}{T^2}
     \Big]W_\K(T)                                              \notag\\
&
    +\Big[\frac{2}{\kappa} - \frac{\kappa}{2}
        + \frac{8.9250\kappa^2 + 3\kappa + 74.4076}{2\kappa T}
        + \frac{1.7702\kappa + 27.9029}{T^2}
     \Big]\nK                                                  \notag\\
&
   + \frac{4.3271\kappa^2 + 34.6168}{\kappa T}
   + \frac{1.2982\kappa + 25.9626}{T^2},                       \notag
\end{align}
%
which is an analogous of \cite[Equation~(4.3)]{GrenieMolteni2}.

\noindent For $h<0$, we take $h=-\frac{\kappa x}{T}$, we then have $x+h>1$ if
$\kappa\leq 3$, $x\geq 3$ and $T\geq 5$. We slightly modify the bound for $A$ in
that case and take
\[
A  =  1+\Big(1-\frac{\kappa}{T}\Big)^{\frac{3}{2}}
 \leq 2-\frac{3\kappa}{2T}+\frac{\kappa^2}{2T^2}.
\]
We still have $|w_\rho|\leq \frac{1}{2}+\frac{\kappa}{6T}$ and Equation~(4.4)
of \cite{GrenieMolteni2} becomes for $T\geq 5$
\begin{align}
\frac{\pi}{\sqrt{x}}
  \Big|\sum_{\rho}&\frac{(x+h)^{\rho+1}-x^{\rho+1}}{h\rho(\rho+1)}
       -\sum_{|\gamma|<T}\frac{x^\rho}{\rho}
  \Big|                                                            \label{eq:A2.3}\\
\leq&
   \Big[\frac{2}{\kappa}+\frac{\kappa}{2}
        +\frac{\kappa^3/3 + 1.4427\kappa^2 -3\kappa + 11.5416}{2\kappa T}
        +\frac{0.2405\kappa^2 + 0.7886\kappa - 4.3281}{T^2}
   \Big]W_\K(T)                                                   \notag\\
&
 + \Big[\frac{2}{\kappa}-\frac{\kappa}{2}
       + \frac{-\kappa^3/3 + 8.9250\kappa^2 - 3\kappa + 74.4076}{2\kappa T}
       + \frac{1.4875\kappa^2 + 2.3602\kappa - 27.9028}{T^2}
     \Big]\nK                                                     \notag\\
&
   + \frac{4.3271\kappa^2 + 34.6168}{\kappa T}
   + \frac{1.4424\kappa^2+1.7309\kappa - 25.9626}{T^2}.           \notag
\end{align}
Let $M_{W,\pm}(T)$, $M_{n,\pm}(T)$ and $M_{c,\pm}(T)$ be the functions of $T$
such that the right hand side of \eqref{eq:A2.2} and \eqref{eq:A2.3} respectively\
are
\begin{align*}
&M_{W,+}(T)W_\K(T) + M_{n,+}(T) \nK + M_{c,+}(T), \\
&M_{W,-}(T)W_\K(T) + M_{n,-}(T) \nK + M_{c,-}(T),
\end{align*}
and their differences let be denoted as
\[
\let\ds=\displaystyle
\begin{array}{l@{}c@{}l@{}c@{}l}
D_W(T)&{}:={}& M_{W,+}(T)-M_{W,-}(T)
      &{} ={}& \ds\frac{18-\kappa^2}{6T}
          + \frac{8.6562-0.1971\kappa-0.2405\kappa^2}{T^2},                   \\[1ex]
D_n(T)&{}:={}& M_{n,+}(T)-M_{n,-}(T)
      &{} ={}& \ds\frac{18+\kappa^2}{6T}
          + \frac{55.8057 - 0.5900\kappa - 1.4875\kappa^2}{T^2},              \\[1ex]
%
D_c(T)&{}:={}& M_{c,+}(T)-M_{c,-}(T)
      &{} ={}& \ds\frac{51.9252 - 0.4327\kappa - 1.4424\kappa^2}{T^2}.
\end{array}
\]
%
We then have
\begin{multline*}
\Big|\psi_\K(x)-x+\sum_{|\gamma|<T}\frac{x^\rho}{\rho}\Big|
\leq \frac{\sqrt{x}}{\pi}\big(M_{W,+}(T)W_\K(T) + M_{n,+}(T) \nK + M_{c,+}(T)\big)       \\
     + \frac{\kappa x}{2T}
     + |r_\K|
     + \delta_{(r_1,r_2),(1,0)}\log(1-x^{-2})
     + \delta_{(r_1,r_2),(0,1)}\log(1-x^{-1}) \\
     + \max\Big(0,(r_1+r_2-1)\log x - \frac{\sqrt{x}}{\pi}\big(D_{W}(T)W_\K(T) + D_{n}(T) \nK + D_{c}(T)\big)\Big).
\end{multline*}
The claim follows if $d_\K=0$ since $D_W(T)$, $D_n(T)$ and $D_{c}(T) \geq 0$
($D_c$ and the coefficients of $\frac{1}{T^2}$ in $D_n$ and $D_W$ are positive
$\forall\kappa\in[-6,4]$). If $d_\K>0$ we have
$\frac{1}{\nK}\log\disc\geq \frac{1}{2}\log5$ thus
\begin{multline*}
D_{W}(T)\frac{W_\K(T)}{\nK} + D_{n}(T)
\geq                                                                         \\
\Big(\frac{18-\kappa^2}{6T} + \frac{8.6562-0.1971\kappa-0.2405\kappa^2}{T^2}\Big)
\log\Big(\frac{\sqrt{5}T}{2\pi}\Big)
+ \frac{18+\kappa^2}{6T} + \frac{55.8057 - 0.5900\kappa - 1.4874\kappa^2}{T^2}\\
\geq
\frac{3.6133\pi}{T}
\end{multline*}
when $T\geq 5$ and $0\leq\kappa\leq 2$.
\end{proof}

\section{Proof of Theorem~\ref{th:B1.1}}
By Equation~\eqref{eq:A2.1c} we have for $T\geq 5$
\begin{equation}\label{eq:A3.1}
\Big|\sum_{\substack{\rho\\|\gamma|<T}}\frac{x^\rho}{\rho}\Big|
\leq
\frac{\sqrt{x}}{\pi}\Big[
  \Big(\log\Big(\frac{T}{2\pi}\Big) +\alpha\Big)\log \disc
+ \Big(\frac{1}{2}\log^2\Big(\frac{T}{2\pi}\Big)+\beta\Big)\nK
+ \gamma
\Big]
\end{equation}
with $\alpha=3.9792$, $\beta=-1.4969$ and $\gamma=25.5362$.
Recalling the upper bound for $|r_\K|$ in Lemma~\ref{lem:A2.1}, from the result
in Theorem~\ref{th:B2.5} and~\eqref{eq:A3.1} we deduce that for $x\geq3$ and
$T\geq5$,
\begin{equation}\label{eq:A3.2}
|\psi_\K(x)-x|\leq
      \Big(\frac{\sqrt{x}}{\pi}F(T)+1.0155\Big)\log\disc
    + \Big(\frac{\sqrt{x}}{\pi}G(T)-2.1042\Big)\nK
    + H(x,T)
\end{equation}
with
\begin{equation}\label{eq:A3.3}
\begin{array}{r@{}>{{}}c<{{}}@{}l}
F(T)  &:=&\displaystyle
           \log\Big(\frac{T}{2\pi}\Big)
           + \frac{2}{\kappa}+\frac{\kappa}{2}
           + \frac{1.4427\kappa^2+3\kappa+11.5416}{2\kappa T}
           + \frac{0.5915\kappa+4.3282}{T^2}
           + \alpha,                                                \\[.3cm]
G(T)  &:=&\displaystyle
           \frac{1}{2}\log^2\Big(\frac{T}{2\pi}\Big)
           + \Big(\frac{2}{\kappa}+\frac{\kappa}{2}
                + \frac{1.4427\kappa^2+3\kappa+11.5416}{2\kappa T}
                + \frac{0.5915\kappa+4.3282}{T^2}
             \Big)\log\Big(\frac{T}{2\pi}\Big)                      \\[.3cm]
      &  &\displaystyle
           + \beta
           + \frac{2}{\kappa} - \frac{\kappa}{2}
           + \frac{8.9250\kappa^2 + 3\kappa + 74.4076}{2\kappa T}
           + \frac{1.7702\kappa + 27.9029}{T^2},                    \\[.3cm]
H(x,T)&:=&\displaystyle
           \frac{\kappa x}{2T}
           + \frac{\sqrt{x}}{\pi}
             \Big[
                     \gamma
                   + \frac{(1.3774\kappa^2 + 11.0190)\pi}{\kappa T}
                   + \frac{(0.4133\kappa + 8.2643)\pi}{T^2}
             \Big]
           + 8.3423
           + \epsilon_\K(x,T).
\end{array}
\end{equation}
%
We need to choose $T$ to get the lowest possible bound for $|\psi_\K(x)-x|$,
thus we choose the best $T$ by looking for an approximate zero of
\[
\frac{\partial}{\partial T}
    \Big(\frac{\sqrt{x}}{\pi}F(T)\log\disc
    + \frac{\sqrt{x}}{\pi}G(T)\nK
         + H(x,T)
    \Big)
\]
above $5$. Unfortunately we are not able to find $T$ as an explicit function
of $x$.\\
\noindent{}Both $T\mapsto F(T)$ and $T\mapsto G(T)$ have a unique minimum while
$T\mapsto H(x,T)$ is decreasing for any $x>0$. The main increasing terms are
$\frac{\sqrt{x}}{2\pi}\big[\log\big(\frac{T}{2\pi}\big)+\frac{\kappa}{2}+\frac{2}{\kappa}\big]^2\,\nK$
from $G(T)$, and
$\frac{\sqrt{x}}{\pi}\big[\log\big(\frac{T}{2\pi}\big)+\frac{\kappa}{2}+\frac{2}{\kappa}\big]\log\disc$
from $F(T)$, while the main decreasing term is $\frac{\kappa x}{2T}$ from
$H(x,T)$. The derivative of the sum of these three terms is zero for
$T\log\big(\frac{e^{\frac{2}{\kappa}+\frac{\kappa}{2}}\rdisc
T}{2\pi}\big)=\frac{\kappa\pi\sqrt{x}}{2\nK}$, we should thus choose
\[
T    =
T_W := \frac{2\pi}{\rdisc e^{\frac{2}{\kappa}+\frac{\kappa}{2}}}e^{W\Big(\frac{\kappa e^{\frac{2}{\kappa}+\frac{\kappa}{2}}\rdisc\sqrt{x}}{4\nK}\Big)}
     = \frac{\kappa\pi\sqrt{x}}{2\nK W\Big(\frac{\kappa e^{\frac{2}{\kappa}+\frac{\kappa}{2}}\rdisc\sqrt{x}}{4\nK}\Big)}.
\]
However, for $\rdisc\to\infty$ we have $T_W\to 0$. We thus slightly complicate
the expression we are trying to minimize: this will have the effect to give a
minimum that is both more precise and above $5$. The expressions contain the
parameter $\kappa$, which has to be fixed. To find a good value for $\kappa$,
we computed the asymptotic expansion of the result with optimal $T$ and
$\kappa$ unevaluated but independent of $x$. This is
\[
  \frac{\sqrt{x}}{2\pi}
    \Big(
        \frac{\log^2 x}{4}
      - \log x\log\log x
      + \Big(\log\rdisc+\log\Big(\frac{\kappa e^{\frac{2}{\kappa}+\frac{\kappa}{2}}}{2\pi\nK}\Big)+1\Big)\log x
      + o(\log x)
    \Big)\nK,
\]
so that the best value for $\kappa$ is the one minimizing
$\kappa e^{\frac{2}{\kappa} +\frac{\kappa}{2}}$, i.e. $\sqrt{5}-1$.
Thus, to ease a little bit the computations, we set $\kappa=\sqrt{5}-1$ right
now, and we retain the symbol $\kappa$ only in those terms which will
contribute to the main part of the result. Notice that
$\frac{2}{\kappa}+\frac{\kappa}{2}=\sqrt{5}$ and
$\frac{2}{\kappa}-\frac{\kappa}{2}=1$ and that~\eqref{eq:A3.2}
and~\eqref{eq:A3.3} give \begin{equation*}
\begin{array}{r@{}>{{}}c<{{}}@{}l}
F(T)  &\leq&\displaystyle
          \log\Big(\frac{T}{2\pi}\Big)
          + \frac{2}{\kappa}+\frac{\kappa}{2}
          + \frac{7.0604}{T}
          + \frac{5.0593}{T^2}
          + 3.9792,                                         \\[.3cm]
G(T)  &\leq&\displaystyle
          \frac{1}{2}\Big(\log\Big(\frac{T}{2\pi}\Big)
                         +\frac{2}{\kappa}+\frac{\kappa}{2}
                     \Big)^2
          + \Big(
            \frac{7.0604}{T}
          + \frac{5.0593}{T^2}
            \Big)\log\Big(\frac{T}{2\pi}\Big)               \\[.3cm]
       &   &\displaystyle
          - 2.9969
          + \frac{37.1145}{T}
          + \frac{30.0910}{T^2},                            \\[.3cm]
H(x,T)&\leq&\displaystyle
            \frac{\kappa x}{2T}
          + \frac{\sqrt{x}}{\pi}\Big(25.5362
                                   + \frac{33.3542}{T}
                                   + \frac{27.5673}{T^2}
                                \Big)
          + 8.3423
          + \epsilon_\K(x,T).
\end{array}
\end{equation*}
We have kept $\kappa$ in all terms which will contribute to the highest order
terms of the asymptotic expansion in $x$, in order to make explicit the role
of this parameter on the final quality of the result.\\
Since $\epsilon_\K$ is small with respect to most other parameters and not
differentiable, we remove it from the optimization process. Let then
\enlargethispage{-3\baselineskip}
\begin{align}
E_0(x,T):=& F(T)\log\rdisc + G(T)
          + \frac{\pi}{\nK\sqrt{x}}(H(x,T)-\epsilon_\K(x,T))  \notag\\
      \leq& \Big(
              \log\Big(\frac{T}{2\pi}\Big)
            + \frac{2}{\kappa}+\frac{\kappa}{2}
            + \frac{7.0604}{T}
            + \frac{5.0593}{T^2}
            + 3.9792
              \Big)\log\rdisc                                 \notag\\
          & + \frac{1}{2}\Big(\log\Big(\frac{T}{2\pi}\Big)
                           +\frac{2}{\kappa}+\frac{\kappa}{2}
                       \Big)^2
            + \Big(
              \frac{7.0604}{T}
            + \frac{5.0593}{T^2}
              \Big)\log\Big(\frac{T}{2\pi}\Big)               \notag\\
          & - 2.9969
            + \frac{37.1145}{T}
            + \frac{30.0910}{T^2}
            + \frac{\kappa\pi\sqrt{x}}{2\nK T}
            + \frac{1}{\nK}\Big(25.5362
                              + \frac{33.3542}{T}
                              + \frac{27.5673}{T^2}
                           \Big)
            + \frac{8.3423\pi}{\nK\sqrt{x}}                   \notag\\
      \leq& \frac{1}{2}\Big(\log\Big(\frac{T}{2\pi}\Big)
                           +\frac{2}{\kappa}+\frac{\kappa}{2}
                           + \log\rdisc
                       \Big)^2
            - \frac{1}{2}\log^2\rdisc                         \notag\\
          & + \Big(\frac{7.0604}{T}
                + \frac{5.0594}{T^2}
              \Big)
              \Big(\log\Big(\frac{T}{2\pi}\Big)
                  +\frac{2}{\kappa}+\frac{\kappa}{2}
                  + \log\rdisc
              \Big)                                           \notag\\
          & + \frac{\kappa\pi\sqrt{x}}{2\nK T}
            + \frac{21.3270}{T}
            + \frac{18.7781}{T^2}
            + \frac{1}{\nK}\Big(\frac{33.3542}{T}
                              + \frac{27.5673}{T^2}
                           \Big)                              \notag\\
          & + 3.9792\log\rdisc
            - 2.9969
            + \frac{25.5362}{\nK}
            + \frac{8.3423\pi}{\nK\sqrt{x}}                   \notag\\
        =:& E(x,T).                                           \label{eq:A3.4}
\end{align}
It is obvious that $\lim_{T\to\infty} E(x,T)=\infty$. We have
\begin{align*}
\frac{\partial E(x,T)}{\partial T}
=& \Big(\frac{1}{T}
      - \frac{7.0604}{T^2}
      - \frac{10.1186}{T^3}
   \Big)
   \Big(\log\Big(\frac{T}{2\pi}\Big)
       +\frac{2}{\kappa}+\frac{\kappa}{2}
       + \log\rdisc
   \Big)                                                      \\
 & - \Big[\frac{14.2666}{T^2}
        + \frac{32.4969}{T^3}
        + \frac{1}{\nK}\Big(\frac{33.3542}{T^2}
                          + \frac{55.1346}{T^3}
                       \Big)
        + \frac{\kappa\pi\sqrt{x}}{2\nK T^2}
   \Big].
\end{align*}
Let $T_F=8.282137\dots$ be the positive root of
$T^2 - 7.0604T
     -  10.1186
$
(which is where the estimate of $F$ reaches its minimum). We obviously have
$\frac{\partial E(x,T)}{\partial T}|_{T=T_F}<0$ and
$\frac{\partial E(x,T)}{\partial T}=0$ when
\begin{multline}\label{eq:A3.5}
   \Big(T
      - 7.0604
      - \frac{10.1186}{T}
   \Big)
   \Big(\log\Big(\frac{T}{2\pi}\Big)
       +\frac{2}{\kappa}+\frac{\kappa}{2}
       + \log\rdisc
   \Big)                                                      \\
= \frac{\kappa\pi\sqrt{x}}{2\nK}
+ 14.2666
+ \frac{32.4969}{T}
+ \frac{1}{\nK}\Big(33.3542
                    + \frac{55.1346}{T}
               \Big).
\end{multline}
The left hand side of this equation is increasing for $T>T_F$ and maps
$[T_F,{+\infty})$ onto $[0,{+\infty})$ while the right hand side is decreasing
for $T>0$ thus the equation has a single solution for $T>T_F$. Thus for given
$\K$ and $x$, $E(x,T)$ has a single local minimum for some $T>T_F$ and this
minimum is reached for the unique $T>T_F$ satisfying~\eqref{eq:A3.5}. The
solutions (in $T$) of~\eqref{eq:A3.5} can unfortunately not be expressed
with standard analytic functions. We thus slightly modify~\eqref{eq:A3.5} to
have a solution with a nice expression in terms of the Lambert-$W$ function.
We will discuss in Remark~\ref{rk:A1} below the effect of the change we
made to the equation.

\noindent
Suppose we have found a $T_0$ satisfying
\begin{equation}\label{eq:A3.6}
   \Big(T_0
      - 7.0604
      - \frac{10.1186}{T_0}
   \Big)
   \Big(\log\Big(\frac{T_0}{2\pi}\Big)
       +\frac{2}{\kappa}+\frac{\kappa}{2}
       + \log\rdisc
   \Big)                                                      \\
= \frac{\kappa\pi\sqrt{x}}{2\nK}
+ 21.3270
+ \frac{33.5251}{\nK}.
\end{equation}
Denote
\[
\lamb := \log\Big(\frac{T_0}{2\pi}\Big)
         +\frac{2}{\kappa}+\frac{\kappa}{2}
         + \log\rdisc
\]
so that
\[
\frac{\kappa\pi\sqrt{x}}{2T_0\nK}
+ \frac{21.3270}{T_0}
+ \frac{33.3542}{T_0\nK}
                      = \Big(1
                            - \frac{7.0604}{T_0}
                            - \frac{10.1186}{T_0^2}
                        \Big)\lamb.
\]
Then
\begin{align}
E(x,T_0)
           =& \frac{1}{2}\lamb^2
              + \lamb
              - \frac{1}{2}\log^2\rdisc
              + \frac{18.7781
                    + \frac{27.5673}{\nK}
                    - 5.0593w}{T_0^2}                   \label{eq:A3.7}\\
            & + 3.9792\log\rdisc
              - 2.9969
              + \frac{25.5362}{\nK}
              + \frac{8.3423\pi}{\nK\sqrt{x}}           \notag
\intertext{and so, according to Lemma~\ref{lem:A3.1}, see below, it is}
        \leq& \frac{1}{2}(\lamb+1)^2
              - \frac{1}{2}\log^2\rdisc
              + 3.9792\log\rdisc
              - 3.4969
              + \frac{25.5362}{\nK}
              + \frac{8.8590\pi}{\nK\sqrt{x}}           \notag\\
            =& \frac{1}{2}\Big(\log\Big(\rdisc e^{\frac{2}{\kappa} + \frac{\kappa}{2}}\frac{T_0}{2\pi}\Big)
                            + 1
                         \Big)^2
              - \frac{1}{2}\log^2\rdisc                 \label{eq:A3.8}\\
            & + 3.9792\log\rdisc
              - 3.4969
              + \frac{25.5362}{\nK}
              + \frac{8.8590\pi}{\nK\sqrt{x}}.          \notag
\end{align}
To have an upper-bound for $E(x,T_0)$, we can substitute $T_0$
in~\eqref{eq:A3.8} by anything greater than $T_0$. We define $T_W$ and
redefine $\lamb$ by
\begin{align*}
a     &:= \frac{\kappa\pi\sqrt{x}}{2\nK}
      + 21.3270
      + \frac{33.3542}{\nK},                                                \\
\lamb &:= W\Big(\frac{e^{\frac2{\kappa}+\frac{\kappa}2}}{2\pi}\rdisc a\Big),\\
T_W   &:= \frac{2\pi e^{\lamb}}{\rdisc e^{\frac{2}{\kappa}+\frac{\kappa}{2}}}
     = \frac{a}{\lamb},
\end{align*}
which means that $T_W$ is the solution of the equation
\begin{equation}\label{eq:A3.9}
   T
   \Big(\log\Big(\frac{T}{2\pi}\Big)
       +\frac{2}{\kappa}+\frac{\kappa}{2}
       + \log\rdisc
   \Big)
= \frac{\kappa\pi\sqrt{x}}{2\nK}
+ 21.3270
+ \frac{33.3542}{\nK}.
\end{equation}
Recalling the constant $T_F$ defined above, $T_F+T_W$ is larger than $T_0$.
Indeed, if we replace $T_0$ by $T_F+T_W$ in~\eqref{eq:A3.6}, the first factor
is bigger than $T_W$ while the second is bigger than the one in~\eqref{eq:A3.9}
so that the left hand side of~\eqref{eq:A3.6} is bigger than its right hand
side; since the left hand side is increasing this proves that $T_F+T_W\geq
T_0$. We now replace $T_0$ by $T_F+T_W$ in~\eqref{eq:A3.8} obtaining
\begin{align*}
E_0(x,T_0)\leq&\frac{1}{2}\Big(\log\Big(\rdisc e^{\frac{2}{\kappa}+\frac{\kappa}{2}}\frac{T_W+T_F}{2\pi}\Big)+1\Big)^2
                + 3.9792\log\rdisc
                - 3.4969
                + \frac{25.5362}{\nK}
                + \frac{8.8590\pi}{\nK\sqrt{x}}          \\
          \leq&\frac{1}{2}\log^2\Big(e^{\lamb+1}+33.5251\rdisc\Big)
                + 3.9792\log\rdisc
                - 3.4969
                + \frac{25.5362}{\nK}
                + \frac{8.8590\pi}{\nK\sqrt{x}}
\end{align*}
which is exactly the first claim in Theorem~\ref{th:B1.1}.
\medskip

We now proceed for the second inequality~\eqref{eq:A1.2}. We fix a value for
$T$, postponing to Remark~\ref{rk:A2} the reason for this choice. The
minimal value for $F$ is reached when $(\kappa,T)\simeq(2.141,7.2773)$ and the
actual value is $\leq 2.2367\pi$. We make a slightly different choice, which is
$\kappa=2$ and $T=10$, which increases slightly the coefficient of
$\sqrt{x}\log\disc$ but decreases the coefficient of $\sqrt{x}\,\nK$ and $x$,
and makes the formula slightly nicer. We then have
\begin{align*}
F(T) &\leq 2.2543\pi \\
G(T) &\leq 0.9722\pi
\end{align*}
which gives
\begin{align*}
|\psi_\K(x)-x|
\leq& (2.2543\sqrt{x}+1.0155)\log\disc + (0.9722\sqrt{x}-2.1042)\nK \\
&   + \frac{x}{10}
    + 9.0458\sqrt{x}
    + 7.0320
    + \epsilon_\K(x,10)
\end{align*}
proving the second claim in Theorem~\ref{th:B1.1}.

\begin{remark}\label{rk:A1}
We discuss some choices we made for the first bound. Let us call $T_{\min}$ the
zero of~\eqref{eq:A3.5} above $T_F$. When we define $T_0$ from~\eqref{eq:A3.6}
we obviously have $T_0\neq T_{\min}$. However the difference between the
functions appearing on the right hand side of \eqref{eq:A3.6} and
\eqref{eq:A3.5} is
\[
7.0604 - \frac{32.4969}{T} - \frac{55.1346}{\nK T}.
\]
This means that, to obtain $T_{\min}$, we should remove from the right hand
side of the equation defining $T_0$ a quantity that is asymptotic to $7.0604$.
Hence, to the first order for $x\to{+\infty}$,
$T_0-T_{\min}\sim \frac{7.0604}{\log T_0}$. Thus
$\log T_0-\log T_{\min}\sim \frac{7.0604}{T_0\log T_0}$, and
$\frac{1}{2}\log^2 T_0-\frac{1}{2}\log^2 T_{\min}\sim \frac{7.0604}{T_0}$.
This difference produces a term of order
$\frac{\nK\sqrt{x}}{T_0}\asymp \nK^2\log x$ which is already much smaller
than the main terms of the upper bound. On the other hand, when we substitute
\[
\frac{\kappa\pi\sqrt{x}}{2\nK T_{\min}}+21.3270+\frac{33.3542}{\nK}
\]
in~\eqref{eq:A3.4} to obtain the equivalent of~\eqref{eq:A3.7}, a term
$\frac{7.0604}{T_{\min}}\sim \frac{7.0604}{T_0}$
would remain. This term cancels out the previous one, so that the final effect
of the replacement of $T_{\min}$ by $T_0$ is even smaller than $\nK^2\log x$.

\noindent{}The situation where $\rdisc\to{+\infty}$ with $x$ and $\nK$ fixed
is slightly different. In this case $T_{\min}$ and $T_0$ both tend to $T_F$ so
that the result is changed by a term of the order of $\frac{1}{\log\rdisc}$.
However~\eqref{eq:A1.1} is not very good anyway because in several steps we
dropped terms in $-\frac{1}{T^2}$, including a term in $-\frac{\lamb}{T^2}$,
which now do not tend to $0$.
\end{remark}
\begin{remark}\label{rk:A2}
One should keep in mind that~\eqref{eq:A1.2} is thought for fixed $x$ and
diverging $\rdisc$. To prove it, we chose $T=10$. We could have used a
technique similar to the one we used for~\eqref{eq:A1.1} optimizing
$T$ in terms of $\rdisc$. However this is not worth it, because we would
obtain a bound which differs from~\eqref{eq:A1.2} by
$(-c_1+o(1))\frac{\sqrt{x}}{\log\rdisc}$. Meanwhile, all the
approximations and choices we made affect the coefficient of
$\sqrt{x}\log\rdisc$ in~\eqref{eq:A1.2}, thus to improve the bound it is more
efficient, for instance, to increase the minimal value above which
Theorem~\ref{th:B2.5} is valid or to refine the bounds in Lemma~\ref{lem:A2.3}.
\end{remark}

\begin{lemma}\label{lem:A3.1}
In the settings of the proof of Theorem~\ref{th:B1.1},
\[
  R := \frac{18.7781+\frac{27.5673}{\nK}- 5.0593\lamb}{T_0^2}
\leq
     \frac{0.5167\pi}{\nK\sqrt{x}}.
\]
\end{lemma}
\begin{proof}
Letting $L:=\frac{2}{\kappa}+\frac{\kappa}{2}+\log\rdisc$ and $S:=
\frac{\kappa\pi\sqrt{x}}{2\nK}
+ 21.3270
+ \frac{33.3542}{\nK}$,
we rewrite~\eqref{eq:A3.6} as
\[
  f(T_0,L)=S.
\]
Using the implicit function theorem, it is then easy to see that
\[
  \frac{\partial\lamb}{\partial L}=
  \frac{\big(T_0+\frac{10.1186}{T_0}\big)
        \big(\log\big(\frac{T_0}{2\pi}\big)+L\big)
       }
       {\big(T_0+\frac{10.1186}{T_0}\big)
        \big(\log\big(\frac{T_0}{2\pi}\big)+L\big)
      +
        \big(T_0
           - 7.0604
           - \frac{10.1186}{T_0}
        \big)
       }
  \geq 0
\]
and that
$\frac{\partial T_0}{\partial x}\geq 0$ and thus $\lamb$ is increasing
with both $L$ and $x$. It is thus obvious that $R$ is decreasing with $x$. We
now prove that, if $\nK\geq 2$ and $x$ is fixed, the maximum value of $R$ is
obtained for the minimal discriminant. Indeed we observe that
\[
\frac{\partial R}{\partial L}
=
\frac{2\big(T_0-7.0604-\frac{10.1186}{T_0}\big)
  \big(18.7781
  + \frac{27.5673}{\nK}\big)
  - 5.0593\lamb(3T_0-2\cdot7.0604-\frac{10.1186}{T_0})
}
       {T_0^2\big(\big(T_0+\frac{10.1186}{T_0}\big)
        \big(\log\big(\frac{T_0}{2\pi}\big)+L\big)
      +
        \big(T_0
           - 7.0604
           - \frac{10.1186}{T_0}
        \big)\big)
       }
\]
which is negative if
\[
\log\rdisc\geq
\frac{2\big(T_0-7.0604-\frac{10.1186}{T_0}\big)
  \big(18.7781
  + \frac{27.5673}{\nK}\big)
}{5.0593\big(3T_0-2\cdot7.0604-\frac{10.1186}{T_0}\big)
       }-\log\Big(\frac{T_0}{2\pi}\Big)-\sqrt{5}.
\]
This is true because as a function of $T_0$ the right hand side has a maximum
value equal to $0.1366\ldots$ (attained for $\nK=2$ and $T_0\approx 21.2153$)
while $\log\rdisc\geq\frac{1}{2}\log 3$.

\noindent{}For each degree, we thus just need to bound $R$ for the
field with minimal absolute discriminant. For the first few $\nK$ we determine
the lowest possible value for $\disc$ using the ``megrez'' number field
tables~\cite{MegrezTables} and for $\nK\geq 8$ we use Odlyzko's Table 3
in~\cite{OdlyzkoTables}.

For any increasing sequence $(x_n)$ let
$c_n:=\frac{\nK\sqrt{x_{n+1}}}{\pi}R(x_n)$. Since $R$ is decreasing in $x$ and
$(x_n)$ is increasing, if $c_{\max}:=\max c_n$, we have
\[
\forall x,\quad R(x)\leq \frac{c_{\max}\pi}{\nK\sqrt{x}}.
\]
We use the sequence $(x_n)$ defined as follows:
\[
x_1     := 3, \qquad
x_{n+1} := x_n + \begin{cases}
                   1 & \text{if $x_n \in [3,5000)$}    \\
                  10 & \text{if $x_n \in [5000,10^4)$} \\
                 100 & \text{if $x_n \in [10^4,10^5)$} \\
                 x_n & \text{otherwise.}
                 \end{cases}
\]
The table below shows the values of $c_{\max}$ for each degree, the
point $x_{n_{\max}}$ where it is reached and the total number of points we
compute (we stop as soon as $R(x_n)<0)$.
For $\nK\geq 9$, we used the general formula for $\nK=9$ with $S=21.3270$ and
$T_0=T_F$, which ensures that the result is valid for all $\nK\geq 9$.
\[
\begin{array}[b]{rlrrr}
   \nK & \multicolumn{1}{l}{c_{\max}} & \multicolumn{1}{l}{x_{n_{\max}}} & n_{\text{points}} \\
\hline
     1 & 0.2110 & 2810 & 6411 \\
     2 & 0.4644 & 4350 & 6402 \\
     3 & 0.5167 & 3986 & 6398 \\
     4 & 0.4443 & 2927 & 5809 \\
     5 & 0.1325 &  694 & 4177 \\
     6 & 0.0144 &   63 &  280 \\
     7 & <0     &    3 &    1 \\
     8 & <0     &    3 &    1 \\
\geq 9 & <0     &    3 &    1
\end{array}\qedhere
\]
\end{proof}


\begin{thebibliography}{3}
\providecommand{\natexlab}[1]{#1}
\providecommand{\url}[1]{\texttt{#1}}
\expandafter\ifx\csname urlstyle\endcsname\relax
  \providecommand{\doi}[1]{doi: #1}\else
  \providecommand{\doi}{doi: \begingroup \urlstyle{rm}\Url}\fi
\bibitem{GrenieMolteni1}
L.~Grenié, G.~Molteni \emph{Explicit smoothed prime ideals theorems under GRH},
  Math. Comp. \textbf{85} (2016), no.~300, 1875--1899 , DOI:
  \url{http://dx.doi.org/10.1090/mcom3039}.

\bibitem{GrenieMolteni2}
L.~Grenié,  G.~Molteni \emph{Explicit versions of the prime ideal theorem
  for Dedekind zeta functions under GRH}, Math. Comp. \textbf{85} (2016),
  no.~298, 889--906,
  DOI: \url{http://dx.doi.org/10.1090/mcom3031}.

\bibitem{LagariasOdlyzko}
J.~C. Lagarias and A.~M. Odlyzko, \emph{Effective versions of the {C}hebotarev
  density theorem}, Algebraic number fields: {$L$}-functions and {G}alois
  properties ({P}roc. {S}ympos., {U}niv. {D}urham, {D}urham, 1975), Academic
  Press, London, 1977, pp.~409--464.

\bibitem{OdlyzkoTables}
A.~M. Odlyzko.
\newblock Discriminant bounds.
\newblock \url{http://www.dtc.umn.edu/~odlyzko/unpublished/index.html},
  1976.

\bibitem{MegrezTables}
The PARI~Group, Bordeaux, \emph{megrez number field tables}, 2008, Package
  \emph{nftables.tgz} from \raggedright \url{http://pari.math.u-bordeaux.fr/packages.html}.

\bibitem{PARI2}
The PARI~Group, Bordeaux, \emph{{PARI/GP}, version {\tt 2.6.0}}, 2013,
  from \url{http://pari.math.u-bordeaux.fr/}.

\bibitem{Winckler}
B.~Winckler, \emph{Théorème de {C}hebotarev effectif}, arxiv:1311.5715,
  \url{http://arxiv.org/pdf/1311.5715v1.pdf}, 2013.
\end{thebibliography}

\end{document}